\documentclass[11pt,a4paper]{amsart}

\usepackage[toc,page]{appendix}
\usepackage{lscape}
\usepackage{hyperref} 
\usepackage{multirow}

\usepackage{amsmath,amsthm, amscd, amssymb, amsfonts}
\usepackage{epsfig}
\usepackage{amsmath,amsthm,amssymb, amscd,enumerate}

\numberwithin{equation}{section}

\theoremstyle{plain}
\newtheorem{theorem} {Theorem} [section]
\newtheorem*{theoremSN} {Theorem}

\newtheorem{proposition} [theorem] {Proposition}
\newtheorem*{propositionSN} {Proposition}

\theoremstyle{definition}

\newtheorem{example} [theorem] {Example}

\renewcommand \parallel {/\kern-3pt/}

\newcommand \g {\mathfrak{g}}

\newcommand \h {\mathfrak{h}}

\renewcommand \k {\mathrm{k}}

\newcommand \End {\operatorname{End}}

\newcommand \im {\operatorname{Im}}

  {\centering\normalfont\Large\bfseries}

\begin{document}


\title{Minimal Faithful Representation of the Heisenberg Lie Algebra with Abelian Factor}

\title[Minimal Faithful Representations]{Minimal Faithful Representation of the Heisenberg Lie Algebra with Abelian Factor}
\author{Nadina Elizabeth Rojas}

\address{Current affiliation: FCEFyN Universidad Nacional de C\'ordoba, \newline \indent Ciudad Universitaria, \newline \indent
(5000) C\'ordoba, \newline \indent Argentina}

\email{nrojas@efn.uncor.edu}

\thanks{Fully supported by FCEFyN (Argentina)}



\begin{abstract}
For a finite dimensional Lie algebra
$\g$ over a field
$\k$ of characteristic zero, the
$\mu$-function (respectively
$\mu_{nil}$-function) is defined to be the minimal dimension of
$V$ such that
$\g$ admits a faithful representation (respectively a faithful nilrepresentation) on
$V$. Let
$\h_m$ be the Heisenberg Lie algebra of dimension
$2m + 1$ and let
$\mathfrak{a}_n$ be the abelian Lie algebra of dimension
$n$. The aim of this paper is to compute
$\mu(\h_m \oplus \mathfrak{a}_n)$ and
$\mu_{nil}(\h_m \oplus \mathfrak{a}_n)$ for all
$m,n \in \mathbb{N}$.
\end{abstract}

\maketitle


\section{Introduction}\label{intro}


In this paper, all Lie algebras and representations are finite dimension over a field
$\k$ of characteristic zero. Given a Lie algebra
$\g$, we denote by
$\mu(\g)$ the minimal dimension of a faithful representation and by
$\mu_{nil}(\g)$ the minimal dimension of a faithful nilrepresentation. By Ado's Theorem (or its proof) these numbers are well defined and they are integers  invariants of
$\g$ (see for instance \cite[page 202]{J}). Clearly,
$\mu(\g) \leq \mu_{nil}(\g)$.

In the theory of the affine crystallographic groups and finitely-generated torsion-free nilpotent groups, the invariant
$\mu$ plays an important role (for details and reference see \cite{B2}). The following results are known:

\begin{propositionSN}\cite[Propositions 2.31 and 3.8]{B2}
Let
$G$ be an Lie group of dimension
$n$ with
$\g = \text{Lie}(G)$ the Lie algebra of
$G$. If
$G$ admits a left-invariant affine structure then
$\mu(\g) \leq n + 1$.
\end{propositionSN}


Determining left-invariant affine structures over a Lie group
$G$ is linked to the existence left symmetric structures on the corresponding Lie algebra
$\g$. Milnor stated that all real solvable Lie algebra would
admit a left symmetric structure. There are many articles with positive results in this direction (see
for instance the original paper of Milnor \cite{Mi} and \cite{A, D, GM}). However, Y. Benoist \cite{BY} and Burde and
Grunewald \cite{BG} gave the first examples of nilpotent Lie groups not admitting any left-invariant affine structure,
they constructed nilpotent Lie algebras
$\g$ such that
$\mu(\g) > \dim \g + 1$.

On the other hand, this problem is also related to the representation theory of finitely generated nilpotent groups,
since it is important to have methods for finding integer matrix representations of
small dimension for this class of groups (see \cite{GN, N}).

Computing
$\mu \text{ and } \mu_{nil}$ (or finding bounds for them) for a given Lie algebra is acknowledged
to be a very difficult task and the goal has been achieved for very few families (see for instance \cite{S, BM, CR, KB}).

Let
$\g$ be a Lie algebra. Assume that the center
$\mathfrak{z}(\g)$ is trivial then the adjoint representation is faithful, it follows that
$\mu(\g) \leq \dim \g$. If
$\g$ is
$k$-step nilpotent whit
$k= 2 \text{ or } 3$ then
\begin{equation}\label{eq:Sc}
\mu(\g) \leq \dim \g +1,
\end{equation}
see \cite{Sc}. In general, if
$\g$ is nilpotent this result is not even true. In this case, it is known that
\begin{equation}\label{eq:Burde}
\mu(\g) < \frac{3}{\sqrt{\dim \g}} 2^{\dim \g},
\end{equation}
see \cite{B}.
Since the classification of representations of nilpotent Lie algebras is a \emph{wild} problem,
it reasonable to expect difficulties in obtaining
$\mu(\g)$.

Let
$\mathfrak{a}_n$ be the abelian Lie algebra of dimension
$n$ and let
$\h_m$ be the Heisenberg Lie algebra of dimension
$2m + 1$ with a basis
$\{X_1, \dots, X_m, Y_1, \dots, Y_m, Z\}$ such that the only non-zero brackets are
$$
[X_i, Y_i] = Z.
$$
It is clear that the center of
$\h_m$ is
$\mathfrak{z}(\h_m) = \k \{Z\}$. The following results are know for
$\h_m \text{ and } \mathfrak{a}_n$.

\begin{theorem}\label{h_mAbelian}
Let
$m, n \in \mathbb{N}$
\begin{enumerate}[(1)]
\item \label{h_m}
      $\mu_{nil}(\h_m) = \mu(\h_m) = m + 2$ (see \cite{B});
\item \label{abelian}
      $\mu(\mathfrak{a}_n) = \left\lceil 2 \sqrt{n - 1}\right\rceil$ and
      $\mu_{nil}(\mathfrak{a}_n) = \left\lceil 2 \sqrt{n}\right\rceil$ (see \cite{S}, \cite{J2}, \cite{MM}).
\end{enumerate}
\end{theorem}

Let
$\h_m \oplus \mathfrak{a}_n$ be the Heisenberg Lie algebra with abelian factor
$\mathfrak{a}_n$. The upper bound for
$\mu(\h_m \oplus \mathfrak{a}_n)$ given by equation (\ref{eq:Sc}) is
$$
\mu(\h_m \oplus \mathfrak{a}_n) < 2m + n + 2.
$$

Our main result in this paper are the value of
$\mu$ and
$\mu_{nil}$ for the Lie algebra
$\h_m \oplus \mathfrak{a}_n$ for all
$m,n \in \mathbb{N}$.
In this direction, we prove the following theorem.

\begin{theoremSN}\label{Teo:Principal}
Let
$m \in \mathbb{N}_0 \text{ and } n \in \mathbb{N}$. Then
$$
\mu_{nil}(\h_m \oplus \mathfrak{a}_n) = m + \left\lceil 2 \sqrt{n + 1}\right\rceil \; \text{ and }\;
\mu(\h_m \oplus \mathfrak{a}_n) = m + \left\lceil 2 \sqrt{n} \right\rceil.
$$
\end{theoremSN}

The paper is organized as follows. In \S\ref{Upperbound}, the upper bounds for
$\mu(\h_m \oplus \mathfrak{a}_n) \text{ and } \mu_{nil}(\h_m \oplus \mathfrak{a}_n)$ are obtained by explicitly
constructing faithful representations and faithful nilrepresentations of
$\h_m \oplus \mathfrak{a}_n$ of minimal dimension.
In \S\ref{Lowerbound} provides a detailed exposition of the proof of the lower bound of
$\mu(\h_m \oplus \mathfrak{a}_n) \text{ and } \mu_{nil}(\h_m \oplus \mathfrak{a}_n)$, here we use a theorem of Zassenhaus and
some results obtained in \cite{CR}.

Our result for
$m = 0$ coincides with the statement of Theorem \ref{h_mAbelian}.(\ref{abelian}) and it is straightforward to see that
our value of
$\mu_{nil}$ coincides with the statement of Theorem \ref{h_mAbelian}.(\ref{h_m}) if
$n = 0$.

In general
\begin{equation}\label{eq:sumadirecta}
\mu(\mathfrak{g}_1 \oplus \mathfrak{g}_2) \leq \mu(\mathfrak{g}_1) + \mu(\mathfrak{g_2})
\end{equation}
for arbitrary Lie algebras
$\mathfrak{g}_1, \mathfrak{g}_2$. If
$\g = \g_1 \oplus \dots \oplus \g_r$ is a semisimple Lie algebra over
$\mathbb{C}$, where
$\g_j$ is a simple ideal of
$\g$, then
$\mu(\g) = \mu(\g_1) + \dots + \mu(\g_r)$ (see \cite{BM}). But the equality in (\ref{eq:sumadirecta}) seldom occurs for nilpotent
Lie algebras, for instance
$\mu(\mathfrak{a}_n) = \lceil 2 \sqrt{n - 1}\rceil < \sum_{i=1}^n \mu(\mathfrak{a}_1) = n$. As far as we know, there is no example of two nilpotent Lie algebra
$\g_i$,
$\dim \g_i \geq 2$, such that
$\mu(\g_1 \oplus \g_2) = \mu(\g_1) + \mu(\g_2)$.

In this context, the Lie algebra
$\h_m \oplus \mathfrak{a}_n$ can be viewed as an extreme case in this sense, since
\begin{align*}
\mu_{nil}(\h_m \oplus \mathfrak{a}_n)& = \mu_{nil}(\h_m) + \mu_{nil}(\mathfrak{a}_n) - \epsilon \\
\mu(\h_m \oplus \mathfrak{a}_n) &= \mu(\h_m) + \mu(\mathfrak{a}_n) - \epsilon.
\end{align*}
with
$\epsilon = 1 \text{ or }2$, depending on the value of n.

This paper is used in \cite{R} to calculate the
$\mu$-function and
$\mu_{nil}$-function for all
$6$-dimensional nilpotent Lie algebra over
$\k$.

\begin{section}{A family of representations of $\h_m \oplus \mathfrak{a}_n$}\label{Upperbound}


We start with a simple proposition.

\begin{proposition}\label{nilpotente}
Let
$\g$ be a nilpotent Lie algebra and let
$(\pi, V)$ be a representation of
$\g$. Then
$(\pi, V)$ is faithful if and only if
$(\pi\mid_{\mathfrak{z}(\g)}, V)$ is faithful on
$\mathfrak{z}(\g)$.
\end{proposition}

\begin{proof}
Assume that
$\g$ is
$k$-step nilpotent and
$\ker \pi$ is non trivial and let
$X_0 \in \ker \pi, X_0 \neq 0$.
Since
$(\pi\mid_{\mathfrak{z}(\g)}, V)$ is faithful on
$\mathfrak{z}(\g)$, we know that
$X_0 \notin \mathfrak{z}(\g)$ and thus there exists
$X_1 \in \g$ such that
$[X_0, X_1] \neq 0$. Since
$\ker \pi$ is an ideal of
$\g$ and
$(\pi\mid_{\mathfrak{z}(\g)}, V)$ is faithful,
$[X_0, X_1] \in \ker \pi, [X_0, X_1] \notin \mathfrak{z}(\g)$. Hence, there exits
$X_2 \in \g$ such that
$[[X_0, X_1], X_2] \neq 0$ and hence, as before
$[[X_0, X_1], X_2] \in \ker \pi, [[X_0, X_1], X_2] \notin \mathfrak{z}(\g)$.
We now apply this argument again
$k$-times, to obtain
$X_{k-1} \in \g$ such that
$[[\dots, [[X_0, X_1], X_2], \dots], X_k] \neq 0$ but this contradicts the fact that
$\g$ is
$k$-step nilpotent.

The converse is clear.
\end{proof}

The Heisenberg Lie algebra
$\h_m$ has a canonical faithful representation
$(\pi_0, \k^{m + 2})$ which in terms of the canonical basis of
$\k^{m + 2}$ is given by
\begin{equation}
\pi_0\left(\sum_{i=1}^m x_i X_i + \sum_{i=1}^m y_i Y_i + z Z\right) =
\left[
\begin{smallmatrix}
0 & x_1 & \dots & x_m & z \\
  &     &       &     & y_1 \\
  &     &   0   &     & \vdots \\
  &     &       &     & y_m \\
  &     &       &     & 0
\end{smallmatrix}
\right] \;\;\;\;\;\; x_i, y_i, z \in \k.
\end{equation}

Let
$\{A_1, \dots, A_n\}$ be a basis of
$\mathfrak{a}_n$, and let
$\{Z, A_1, \dots, A_n\}$ be a basis of the center
$\mathfrak{z}(\h_m \oplus \mathfrak{a}_n)$.

We shall now construct a family of nilrepresentations of
$\h_m \oplus \mathfrak{a}_n$ that contains a faithful nilrepresentation of minimal dimension.

For
$a, b \in \mathbb{N}$, define the linear maps
\begin{enumerate}[-]
\item $\tau_a : \k\{X_1, \dots, X_m\} \rightarrow M_{a,m}(\k),\;\; \tau_a\left(\sum_{i=1}^m x_i X_i \right) = \left[
                              \begin{array}{ccc}
                               x_1 & \dots & x_m \\
                                   &   &  \\
                                   & 0  &  \\
                                   &   &
                              \end{array}
                              \right]$.
\item $\tau_b : \k\{Y_1, \dots, Y_m\} \rightarrow M_{m,b}(\k),\;\;\tau_b \left(\sum_{i=1}^m y_i Y_i \right)=
\left[
\begin{array}{ccc}
y_1     &   &  \\
 \vdots & 0  &  \\
y_m     &   &
\end{array}
\right]$.
\item $\tau_{a,b}: \k \{Z, A_1, \dots, A_n\} \rightarrow M_{a,b}(\k)$\\
      $\tau_{a,b}\left(z Z + \sum_{r=1}^n a_r A_r \right)_{ij}=
\begin{cases}
z, & \text{ if } i=j=1 ;\\
a_{j-1}, & \text{ if } i=1, j\geq 2;\\
a_{j-1 + (i-1)b}, & \text{ if } 2 \leq i \leq \min \left\{a, \left\lceil \frac{n - b + 1}{b}\right\rceil + 1 \right\};\\
0, & \text{ otherwise }.
\end{cases}$
For instance, if
$n = 10, a= 5, b= 3$ we have
$$\tau_{4,3}\left(z Z + \sum_{r=1}^{10} a_i A_i\right) = \left[
     \begin{smallmatrix}
     z   & a_1 & a_2 \\
     a_3 & a_4  & a_5 \\
     a_6 & a_7 & a_8 \\
     a_9 & a_{10} & 0 \\
     0  &  0  & 0
     \end{smallmatrix}
     \right].$$
\end{enumerate}

Now, we define the representation
$(\pi_{a,b}, \k^{m + a + b})$ of
$\h_m \oplus \mathfrak{a}_n$, in terms of the canonical basis of
$\k^{m + a + b}$. Let
$X = \sum_{i=1}^m x_i X_i, Y = \sum_{i=1}^m y_i Y_i \text{ and } A = \sum_{i=1}^n a_i A_i$. Then
$\pi_ {a,b}$ is  given by the following block matrix
$$
\pi_{a,b}\left( X + Y + z Z + A\right) = \\
\left[
\begin{array}{ccc}
0 & \tau_a(X) & \tau_{a,b}(zZ + A)\\
  & 0 & \tau_b(Y) \\
  &   & 0
\end{array}
\right].
$$

\begin{example}
Let
$m = 2, n= 4$ and
$X= \sum_{i=1}^2 x_iX_i + \sum_{i=1}^2 y_iY_i + zZ + \sum_{i=1}^4 a_iA_i$;
\begin{enumerate}[a)]
\item $a= 2, b=3$ then
      $\pi_{2,3}\left(X \right) = \left[
                                    \begin{smallmatrix}
                                    0 & 0 & x_1 & x_2 & z & a_1 & a_2 \\
                                    0 & 0 & 0   & 0   & a_3 & a_4 & 0 \\
                                    0 & 0 & 0   & 0   & y_1 & 0 & 0 \\
                                    0 & 0 & 0   & 0   & y_2 & 0 & 0 \\
                                    0 & 0 & 0   & 0   & 0 & 0 & 0 \\
                                    0 & 0 & 0   & 0   & 0  & 0 & 0 \\
                                    0 & 0 & 0   & 0   & 0 & 0 & 0
                                    \end{smallmatrix}
                                    \right]$;
\item $a= 1, b=3$ then
      $\pi_{1,3}\left(X \right) = \left[
                                    \begin{smallmatrix}
                                    0 & x_1 & x_2 & z & a_1 & a_2 \\
                                    0 & 0   & 0   & y_1 & 0 & 0 \\
                                    0 & 0   & 0   & y_2 & 0 & 0 \\
                                    0 & 0   & 0   & 0 & 0 & 0 \\
                                    0 & 0   & 0   & 0 & 0 & 0 \\
                                    0 & 0   & 0   & 0 & 0 & 0
                                    \end{smallmatrix}
                                    \right]$.
\end{enumerate}
\end{example}

By straightforward calculation we have
\begin{enumerate}[-]
\item $\left(\pi_{a,b}, \k^{m + a + b}\right)$ is a nilrepresentation of
      $\h_m \oplus \mathfrak{a}_n$.
\item If
      $a= b= 1$ then
      $\pi_{1,1}\mid_{\h_m}= \pi_0$.
\item By Proposition \ref{nilpotente},
      $\left(\pi_{a,b}, \k^{m + a + b}\right)$ is faithful if and only if
      $\tau_{a,b}$ is injective. Then
      $\left(\pi_{a,b}, \k^{m + a + b}\right)$ is faithful if and only if
      $ab \geq n + 1$.
\end{enumerate}

Since
$$
\min\{a + b : ab \geq n+1\} = \left\lceil 2 \sqrt{n + 1} \right\rceil
$$
for every
$n \in \mathbb{N}$, we obtain the following result.

\begin{proposition}\label{coro:munilhm}
Let
$m, n \in \mathbb{N}$ then
$$
\mu_{nil}(\h_m \oplus \mathfrak{a}_n) \leq m + \left\lceil 2 \sqrt{n + 1}  \right\rceil.
$$
\end{proposition}

Analogously, we introduce a family of representation of
$\h_m \oplus \mathfrak{a}_n$ that contains a faithful representation of minimal dimension.

Let
$X = \sum_{i=1}^m x_i X_i + \sum_{i=1}^m y_i Y_i + z Z + \sum_{i=1}^n a_i A_i \in \h_m \oplus \mathfrak{a}_n$ and let
$(\pi_{a,b}, \k^{m + a + b})$ be the representation of
$\h_m \oplus \mathfrak{a}_n$ that is  given by
$$
\widetilde{\pi}_{a,b}\left(X \right) =
\pi_{a,b}\left(\sum_{i=1}^m x_i X_i + \sum_{i=1}^m y_i Y_i + z Z + \sum_{i=1}^{n-1} a_i A_i\right) + a_n I
$$
where
$I$ is an identity matrix of size
$m + a + b$.

It is easy to see that
$(\widetilde{\pi}_{a,b}, \k^{m + a + b})$ is a representation of
$\h_m \oplus \mathfrak{a}_n$ and
$\widetilde{\pi}$ is faithful if
$ab \geq n$. Then we obtain the following proposition.

\begin{proposition}\label{coro:muhm}
Let
$m, n \in \mathbb{N}$ then
$$
\mu(\h_m \oplus \mathfrak{a}_n) \leq m + \left\lceil 2 \sqrt{n}  \right\rceil.
$$
\end{proposition}
\end{section}


\begin{section}{The lower bound}\label{Lowerbound}


For this section we prove the lower bound of
$\mu(\h_m \oplus \mathfrak{a}_n) \text{ and } \mu_{nil}(\h_m \oplus \mathfrak{a}_n)$.

We will need the following facts:

Let
      $\g$ be a Lie subalgebra of
      $\h_m \oplus \mathfrak{a}_n$ such that
      $Z \notin \g$, from an argument similar to \cite[Lemma 3]{CR} it follows that
      \begin{equation}\label{eq:lowerbound}
      \dim \g \leq m + \dim \g \cap \mathfrak{z}(\h_m \oplus \mathfrak{a}_n).
      \end{equation}
We will also need the following theorem (see \cite[Theorem 4.2]{CR}).

\begin{theorem}\label{Teo:munilh_m}
Let
$V$ be a finite-dimensional vector space and let
$\mathcal{N}$ be a non-zero abelian subspace of
$\End(V)$ consisting of nilpotent operators. Then there exists a linearly independent set
$B = \{v_1, \dots, v_s\} \subset V$ and a decomposition
$\mathcal{N}= \mathcal{N}_1 \oplus \dots \oplus \mathcal{N}_s$, with
$\mathcal{N}_i \neq 0$ for all
$i$, such that the maps
$F_i : \mathcal{N} \rightarrow V$ defined by
$F_i(N) = N(v_i)$ satisfy
\begin{enumerate}[(1)]
\item $F_i\mid_{\mathcal{N}_i}$ is injective for all
      $i= 1, \dots, s$.
\item $\mathcal{N}_j \subset \ker F_i$ for all
      $1 \leq i < j \leq s$;
\item $\mathcal{N}_j V \subset \im F_i\mid_{\mathcal{N}_i}$ for all
      $1 \leq i < j \leq s$.
\end{enumerate}
Furthermore, given a finite subset
$\{N_1, \dots, N_q\}$ of non-zero operators in
$\mathcal{N}$, the vector
$v_1$ can be chosen so that
$N_k(v_1) \neq 0$ for all
$k = 1, \dots,k$.
\end{theorem}

We are now ready to prove the first main result of this section.

\begin{theorem}\label{TeoNil:h_m}
Let
$m, n \in \mathbb{N}$ then
$$
\mu_{nil}(\h_m \oplus \mathfrak{a}_n) \geq m + \left\lceil 2 \sqrt{ n + 1 } \right\rceil.
$$
\end{theorem}

\begin{proof}
Let
$(\pi, V)$ be a faithful nilrepresentation of
$\h_m \oplus \mathfrak{a}_n$ and let
$\mathfrak{z}= \k\{Z\} \oplus \mathfrak{a}_n$. We apply Theorem \ref{Teo:munilh_m} to the subspace
$\mathcal{N}= \pi(\mathfrak{z})$. We obtain a linearly independent set
$B = \{v_1, \dots, v_s\} \subset V$ and a decomposition
$\mathcal{N}= \mathcal{N}_1 \oplus \dots \oplus \mathcal{N}_s$, with
$\mathcal{N}_i \neq 0$ for all
$i$, such that the maps
$F_i : \mathcal{N} \rightarrow V$ defined by
$F_i(N) = N(v_i)$ satisfy $(1), (2) \text{ and } (3)$ of Theorem \ref{Teo:munilh_m}
We additionally require that
$\pi(Z)(v_1) \neq 0$.

Let
$\phi$ be a linear map
$$
\phi : \h_m \oplus \mathfrak{a}_n \rightarrow V, \;\; \phi(X)= \pi(X)v_1.
$$
Note that
$\phi\mid_{\mathfrak{z}} = F_1 \circ \pi$. We claim that
\begin{enumerate}[(i)]
\item $\dim \im \phi + \dim \ker F_1 \geq m +  n + 1$.
\item $\im \phi \cap \k B = 0$, and thus
      $\dim V \geq s + \dim \im \phi$.
\item $n + 1 \leq s \dim \im F_1$, and thus
      $\left\lceil 2\sqrt{n+1} \right\rceil \leq s + \dim \im F_1$.
\end{enumerate}

\medskip

\noindent\emph{Proof of (i).}
      It is clear that
      $\ker \phi$ is a subalgebra of
      $\h_m \oplus \mathfrak{a}_n$ such that
      $Z \notin \ker \phi$. Since
      $\pi(\ker \phi \cap \mathfrak{z}) = \ker F_1$, it follows from (\ref{eq:lowerbound}) that
      $$
      \dim \ker \phi \leq m + \dim \ker F_1.
      $$
      Since
      $\dim \ker \phi + \dim \im \phi =  2m + 1 + n$, we obtain $(i)$.

\medskip

\noindent\emph{Proof of (ii).}
Let
$v \in \im \phi \cap \k B$ then
$v = \sum_{i=1}^s a_i v_i$ and since
$v \in \im \phi$ there is
$X \in \h_m \oplus \mathfrak{a}_n$ such that
$\pi(X)(v_1) = v$. Hence
\begin{equation}\label{eq:1}
\pi(X)(v_1) = \sum_{i=1}^s a_i v_i.
\end{equation}
We must prove that
$a_i = 0$ for all
$i$. Assume that
$a_i \neq 0$ for some
$i$ and let
$i_0 = \max \{i : a_i \neq 0\}$. Since
$\pi(X)$ is nilpotent endomorphism on
$V$, its only eigenvalue is zero and thus
$i_0 > 1$. Let
$N \in \mathcal{N}_0$,
$N \neq 0$ and let us apply
$N$ to both sides of equation (\ref{eq:1}). We obtain zero on the left hand side, since
$N \in \\pi(\mathfrak{z})$,
$i_0 > 1$. But from
$(1) \text{ and } (2)$ of the Theorem \ref{Teo:munilh_m}, we obtain on the right hand side
$a_{i_0} N(v_{i_0}) \neq 0$, which is a contradiction.

\medskip

\noindent\emph{Proof of (iii).} Part
$(1) \text{ and } (3)$ combined imply that
$\dim \mathcal{N}_x \geq \dim \mathcal{N}_x $ if
$x < y$. In particular
$\dim \mathcal{N}_1 \geq \mathcal{N}_j$ for all
$j=1, \dots, s$ and thus
$$
n + 1 = \dim \mathcal{N} = \sum_{j=1}^s \dim \mathcal{N}_j \leq s \dim \mathcal{N}_1 = s \dim \im F_1.
$$
Since
$\min\{a+ b : a, b \in \mathbb{N} \text{ and } ab \geq n+1\}= \left\lceil 2 \sqrt{n+1} \right\rceil$ for all
$n \in \mathbb{N}$, we obtain
$(iii)$.

From
$(i) \text{ and } (ii)$ it follows that
$$
\dim V + \dim \ker F_1 \geq m + n + 1 + s,
$$
and combining it with
$(iii)$ we obtain
$$
\dim V + \dim \ker F_1 + \dim \im F_1 \geq m + n + 1 + \left\lceil 2\sqrt{n+1} \right\rceil.
$$
Finally, since
$\dim \ker F_1 + \dim \im F_1 = n + 1$ we obtain
$$
\dim V \geq m + \left\lceil 2\sqrt{n+1} \right\rceil
$$
and this completes the proof.
\end{proof}

First we recall a theorem due to Zassenhaus (see \cite[page 41]{J}) for to prove the lower bound for
$\mu(\h_m \oplus \mathfrak{a}_n)$.

\begin{theorem}\cite[Theorem 2.1]{CR}
Let
$\g$ be finite dimensional nilpotent Lie algebra and let
$(\pi, V)$ be a finite-dimensional representation of
$\g$. If
$\k$ is algebraically closed then
$$
V= V_1 \oplus \dots \oplus V_s,
$$
such that
$\pi(X)\mid_{V_i}$ is a scalar
$\lambda_i$ plus a nilpotent operator
$N_i(X)$ on
$V_i$ for all
$X \in \g$ and
$i=1, \dots, s$. Moreover,
$(N_i, V_i)$ is a nilrepresentation of
$\g$ for all
$i=1, \dots, s$.
\end{theorem}

We end the paper with the second main result of this section.

\begin{theorem}\label{Teo:h_m}
Let
$m, n \in \mathbb{N}$ then
$$
\mu(\h_m \oplus \mathfrak{a}_n) \geq  m + \left\lceil 2 \sqrt{ n } \right\rceil.
$$
\end{theorem}

\begin{proof}
Let
$(\pi, V)$ a faithful representation of
$\h_m \oplus  \mathfrak{a}_n$, by Zassenhaus theorem we have
$$
V = V_1 \oplus \dots \oplus V_s,
$$
such that
$\pi(X)\mid_{V_i} = \lambda_i(X) I + N_i(X)$
for all
$X \in \g$.

Since
$(\pi, V)$ is a faithful representation there is
$v \in V$ such that
$\pi(Z)(v) \neq 0$.
It follows that there exist
$i_0= 1, \dots, s$, such that
$\pi(Z)(v_{i_0}) \neq 0$.

Since
$Z \in \mathfrak{z}(\h_m \oplus \mathfrak{a}_n)\cap [\h_m \oplus \mathfrak{a}_n, \h_m \oplus \mathfrak{a}_n]$, we have
$\lambda_{i_0}(Z) = 0$ and therefore
$N_{i_0}(Z)(v_{i_0}) \neq 0$.
By Proposition \ref{nilpotente} we have
$(N_{i_0}\mid_{\h_m}, V_{i_0})$ is a faithful nilrepresentation of
$\h_m$. Therefore, if
$\mathfrak{c} = \ker N_{i_0}$, it follows that
$\mathfrak{c} \cap \h_m = 0$ and thus
$\mathfrak{c}$ is an abelian Lie subalgebra of
$\h_m \oplus \mathfrak{a}_n$.
Let
$\mathfrak{b}$ be a complementary subspace of
$\h_m \oplus \mathfrak{c}$ in
$\h_m \oplus \mathfrak{a}_n$, thus
$\h_m \oplus \mathfrak{a}_n = \h_m \oplus \mathfrak{b} \oplus \mathfrak{c}$. It is clear that
$\h_m \oplus \mathfrak{b}$ is a Lie subalgebra of
$\h_m \oplus \mathfrak{a}_n$. By construction of
$\mathfrak{b}$, we have
$(N_{i_0}\mid_{\h_m \oplus \mathfrak{b}}, V_{i_0})$ is a faithful nilrepresentation of
$\h_m \oplus \mathfrak{b}$. We conclude, by Theorem \ref{TeoNil:h_m}, that
\begin{align*}
\dim V_{i_0} &\geq m + \left\lceil 2 \sqrt{ \dim \mathfrak{b} + 1} \right\rceil.
\end{align*}

Let
$\mathfrak{c}_0 = \ker \lambda_{i_0} \cap \mathfrak{c}$ then
$\dim \mathfrak{c}_0 = \dim \mathfrak{c} - 1$ and
$(\pi\mid_{ \mathfrak{c}_0}, \oplus_{i \neq i_0} V_i)$ is a faithful representation of
$ \mathfrak{c}_0$. By Theorem \ref{h_mAbelian}(\ref{abelian}), we have
\begin{align*}
\dim \oplus_{i \neq i_0} V_i &\geq \left\lceil 2 \sqrt{\dim \mathfrak{c}_0 - 1} \right\rceil
\end{align*}
Since
$\dim V =   \dim V_{i_0} + \dim \oplus_{i \neq i_0} V_i $, it follows that
\begin{align*}
\dim V & \geq m + \left\lceil 2 \sqrt{\dim \mathfrak{b} + 1} \right\rceil + \left\lceil 2 \sqrt{\dim \mathfrak{c}_0 - 1} \right\rceil \\
       & \geq m + \left\lceil 2 \sqrt{\dim \mathfrak{b} + \dim \mathfrak{c}_0 + 1} \right\rceil \\
       & \geq m + \left\lceil 2 \sqrt{ n } \right\rceil
\end{align*}
and this completes the proof.
\end{proof}

\end{section}

\section*{Acknowledgments}
The author sincerely thanks Professor Leandro Cagliero of Universidad Nacional de C\'ordoba, Argentina
for the helpful discussions concerning the material in this paper, for reading the draft and making a number of helpful suggestions.
The author also wishes to express his thanks to Gast\'on Garc\'ia for reading the draft and making helpful
suggestions.


\end{document}